\documentclass[12 pt]{amsart}
\usepackage{amsthm, amsfonts, amssymb, color}
\usepackage{mathrsfs}
\usepackage{amsmath}
\usepackage{hyperref}
\usepackage{amstext, amsxtra}
\usepackage{txfonts}
\usepackage{comment}

\allowdisplaybreaks
\usepackage{geometry}
\usepackage{color}

\newcommand{\Red}[1]{{\color{red} #1}}

\newtheorem{theorem}{Theorem}[section]
\newtheorem{proposition}[theorem]{Proposition}

\newtheorem{definition}[theorem]{Definition}

\newtheorem{remark}[theorem]{Remark}

\newcommand{\ra}{\rangle}

\newcommand{\jap}[1]{\langle #1\rangle}
\newcommand{\jx}{\left< x\right>}

\newcommand\R{\mathbb{R}}

\newcommand{\la}{\langle}

 \newcommand{\pwb}{\psi_{w,b}}
\numberwithin{equation}{section}

\title   [] { A General Scattering theory for Nonlinear and Non-autonomous Schr\"odinger Type Equations- A Brief description   }

\author{Baoping Liu}
\address{Beijing International Center for Mathematical Research\\
 Peking University\\
 Beijing\\
  China}
\email{baoping@math.pku.edu.cn}
\author{Avy Soffer}
\address{Department of Mathematics\\
Rutgers University\\
110 Frelinghuysen Rd.\\
Piscataway, NJ, 08854, USA}
\email{soffer@math.rutgers.edu}

\thanks{
B. Liu is supported in part by the NSFC12071010 and NSFC11631002.
A.
Soffer is supported in part by NSF grant DMS-1600749 and by NSFC11671163
}

\begin{document}
\begin{abstract}
We give a short description of the proof of asymptotic-completeness for NLS-type equations, including time dependent potential terms,
with radial data in three dimensions. We also show how the method applies for the two-body Quantum Scattering case.
\end{abstract}

\maketitle

\section{Introduction}
\textit{Coherent structures}, which include \textit{solitons, breathers, kinks, black-holes, vortices...} are solutions to nonlinear PDEs that remain spatially localized for all times.
 They arise as an outcome of balancing between the  linear dispersion  and nonlinear  attractive effect, and seem to be a universal phenomenon in many physical systems such as fluids, plasma, string theory,  supergravity, etc.

Mathematically, these solutions  play a fundamental role in understanding the long time dynamics for general solutions.  In fact, it is conjectured that the generic asymptotic states are given by independently (freely) moving coherent structures and free radiation~\cite{Soffer}. This statement is called \textit{asymptotic completeness}, sometimes also goes by the name \textit{soliton resolution}~\cite{Tao-soliton},  is one of most challenging and exciting topics in dispersive equations~\cite{Tao04, Tao07, Tao08, DKM,DJKM, JiaLiuXu, Generic, KLLS, KLS}.


In~\cite{LiuSoffer}, we study the asymptotic completeness for NLS with rather general nonlinearity and also time dependent potentials.  By constructing new propagation observables and proving propagation estimates,  we were able to construct the free channel wave operators in the spherically symmetric case. This shows that in  the exterior region, solutions behave like free waves. We also   provide more information about the left-over part of the solution. To our best knowledge, previously the only attempt in this direction are the works of Tao~\cite{Tao04, Tao07, Tao08}, where in $3d$ he obtained similar decomposition with error going to $0$ in $\dot{H}^1$ norm,  and in high dimension (and including a trapping potential term) he showed existence of  global compact attractor in $H^1$.  Comparably, we work in $3d$ and our results imply the solutions asymptotically decompose into a free wave and left-over(localized) part in the strong sense in $L^2$. We also provide more detailed information on left-over part.

In this short paper, we present a brief description of our result  and   elaborate the method and ideas therein.
In particular we demonstrate the method by applying it, in a much simplified form, to prove AC (asymptotic-completeness) for two body quantum scattering.

\section{Problem and Results}
We consider the general class of Nonlinear Schr\"odinger type equations of the form:
\begin{equation}\left\{
\begin{aligned}i\partial_t \psi +\Delta \psi & = \mathcal{N}(\psi), \\
\psi(0,x)& =\psi_0\in H^1_{rad}(\R^3)\cap L^2_{rad}(\R^3, |x|dx),
\end{aligned}\right. \label{Main-eq}
\end{equation}
such that
\begin{equation}\sup_{t\in [0,\infty)} \|\psi\|_{H^1(\R^3)} <\infty. \label{global-bound}\end{equation}
The term $\mathcal{N}(\psi)$ represents  a combination of the following cases:
\begin{enumerate}
\item  Defocusing power type nonlinearity with a focusing time independent potential $\mathcal{N}(\psi) = -|\psi|^{p-1}\psi+V(|x|)\phi$, with $p\in (\frac73, 5)$, i.e. it is energy-subcritical and mass-supercritical.

\item The  focusing saturated  nonlinearities as ~\cite{RSS}
$\mathcal{N}(\psi)=- \frac{|\psi|^{m-1}\psi}{1+|\psi|^{m-n}}$, with $m> \frac73, 1<n<\frac73$. Notice that it is mass subcritical  for $|\psi|\gg 1$, and mass supercritical for $|\psi|\ll 1$.

\item The radially symmetric time dependent potential $\mathcal{N}(\psi)=V(t,|x|)\psi$, such that
$|V(t,x)|\leq C (1+|x|)^{-q}$ for all $t$.
\end{enumerate}
In fact, if we write $\mathcal{N}(\psi)= V(|x|, t, |\psi|)\psi$, we only need to require  the following decay assumption for $t\geq 1$,
\begin{equation} \left|F(\frac{|x|}{t^\alpha}\geq 1)V(|x|, t, |\psi|) \right|\lesssim t^{-\beta_0},\label{N-decay}\end{equation}
for some $\beta_0>1.$
Here $F(\lambda)$ is a smooth characteristic function of the interval $[1,+\infty)$, $\alpha \in (\frac13, 1)$ is a parameter chosen later.  Using the radial Sobolev embedding and global $H^1$ bound, we see that this decay assumption is verified for general nonlinearities, and potentials with a suitable choice of the parameters $p, q, m, n$.

Our main result is the following

\begin{theorem}
Let $\phi(t)$ be a global solution to equation (\ref{Main-eq}) satisfying (\ref{global-bound}), then we have the following asymptotic decomposition
\begin{equation}
\lim_{t\rightarrow +\infty}\|\psi(t)- e^{i\Delta t}\Omega_{F}\psi_0 - \psi_{w,b}(t)\|_{L^2(\R^3)}=0.
\end{equation}
Here $\Omega_F$ is the (bounded) nonlinear scattering wave operator, mapping the initial data to the asymptotic free wave,
 $\psi_{w,b}$ is the weakly localized part of the solution, with the following properties
\begin{enumerate}
\item It is weakly localized in the region $|x|\leq t^{\frac12}$,  in the following sense
\[(\pwb, |x|\pwb )\lesssim t^{\frac12}.\]
\item It is smooth, and for $k\geq 1$, \[\|(x\cdot\nabla_x)^k \pwb\| _{L^2_x}\lesssim 1.\] $k\leq K$, $K$ depends on the regularity of the potential term, if present.
\item If the solution is time periodic, then \[ \|x \pwb\|_{L^2_x} \lesssim 1.\]
\end{enumerate}
All the  estimates hold uniformly in time for $t\geq 0.$
\end{theorem}

\section{Phase space operators}Given a self-adjoint operator $A$ which might be time dependent, we denote
$$  \langle A \rangle_t := (A(t)\psi(t), \psi(t))= \int_{\R^3}( A(t)\psi(t) )\overline{ \psi(t)} dx $$ where $\psi$ is a solution to equation (\ref{Main-eq}). By direct computation we get
\begin{align}
\frac{d}{dt}\langle  A \rangle_t
= \langle  i[-\Delta+V , A] +\frac{\partial A}{\partial t} \rangle_t  - 2Im (A\psi, \mathcal{N}(\psi))\label{DtA}
\end{align}
For simplicity, we denote
$$D_HA=i[-\Delta+V, A] +\frac{\partial A}{\partial t} $$
and this operator is called the \textit{Heisenberg derivative}. Notice that if $\psi$ satisfies the free Schr\"{o}dinger equation, then
$\frac{d}{dt}\langle A\rangle = \langle D_H A\rangle.$


For any interval $\Omega\subset \R$, define a smooth function
\begin{equation*}F(\lambda \in \Omega) =\left\{ \begin{aligned} 1, &\hspace{0.5cm}  \mbox{ if } \lambda \in \Omega,   dist(\lambda, \partial \Omega)\geq \epsilon,\\
0,  &  \hspace{0.5cm}\mbox{ if } \lambda \not\in \Omega.
\end{aligned}
\right.
\end{equation*}
We will also require $F(\lambda\in \Omega)$ to be monotonic if $\Omega=[a,+\infty)$ or $(-\infty, a]$.  

Consider the following class of functions,\[
\mathcal{B}_n:= \{f\in C_b^\infty(\R)\left| \int_{\R}|\hat{f}(s)| |s|^{k} ds <\infty, \hspace{0.3cm} \mbox{ for } 1\leq k\leq n\}\right.\]
where $C_b^\infty$ represents class of smooth and bounded functions.  Typical examples in $\mathcal{B}_n$ are the smooth characteristic functions for intervals in $\R$.   For $f\in \mathcal{B}_n$,  and self-adjoint operator $A$
we can define the operator $f(A)$ using spectral calculus. For the convenience of commutator estimate, we also use the representation
\begin{equation}f(A)=\int_{\R}\hat{f}(s)e^{iAs}ds.\label{rep-formula}\end{equation}

 \begin{proposition}~\label{BFA-lemma} Let $A,B$ be self-adjoint operators, and $B$ is bounded.  Let $f\in \mathcal{B}_n $, then we have the following commutator formulas
\begin{align}
[B, f(A)]= & i\int_{\R}   \hat{f}(s) \int_0^s  e^{i(s-u)A}[ B, A]e^{iuA} du ds, \label{BFA-commute} \\
[B, f(A)]= &\sum_{k=1}^{n-1} \frac{1}{k!}f^{(k)}(A)ad_A^{(k)}(B) +R_n, \label{BFA}\\
[B, f(A)]=& \sum_{k=1}^{n-1} \frac{1}{k!}(-1)^{k-1} ad_A^{(k)}(B)f^{(k)}(A) - R^*_n. \label{BFA-adjoint}
\end{align}
Here $ad_A^{(k)}(B)$ is the higher commutators  \[ad_A^{(k)}(B)=[ad_A^{(k-1)}(B), A] = [[B,A],A],\ldots, A],\]
and the remainder term can be estimated as
\[\|R_n\|\leq c_n \|ad_A^{(n)}(B)\| \int |\hat{f}(s)| |s|^{n}ds.\]
\end{proposition}

\section{Asymptotic Completeness of Two body Quantum Scattering}
We illustrate the method we developed by first using some of its aspects to prove AC for the standard two-body short range potential case.
 The problem have been well studied, with several different proofs. We refer to the review paper~\cite{HS-JMP} for more details.

 As a standard procedure, we reduce the two body problem to $-\Delta +V$ on $L^2(\R^N)$.
 Let $H=-\Delta+V(x)$ with $V$ decaying at infinity, and is regular in the sense that commutators with $A:=\frac12(x\cdot p+p\cdot x)$ the dilation generator, are bounded to order 2 :
\begin{align}\label{V}
&V\jx^{\sigma},\, \jx^{\sigma}(x\cdot\nabla) V, \, \jx^{\sigma}(x\cdot\nabla)^2 V \lesssim 1, \sigma>1.\\
&\mathfrak{Im} V=0.
\end{align}
Under these conditions, $H$ is a self adjoint operator, with domain equal to that of the Laplacian.


  Next we recall the statement of asymptotic completeness (AC): For all initial conditions in $L^2,$ the asymptotic solution is known to be one of an explicit form. In the case of two body QM, all solutions are asymptotic to a linear combination of a free wave (solution of the free wave equation) and a bound stable cluster, which is almost periodic in time.
  Moreover, the bound cluster is a linear combination of eigenfunctions of $H$, with time dependent phase.
  The convergence to the asymptotic solution is in $L^2$:

  \begin{align}\label{AC-linear}
 &\lim_{t\to \pm \infty} \|e^{-itH}\psi(0)-e^{+i\Delta t}\phi_{\pm}-P_b\psi_{\pm}(t)\|_{L^2}=0.\\
 &\psi_{\pm}(t)=\sum_j a_j^{\pm}e^{-iE_j t}\psi_j^{\pm}.\\
 &H\psi_j=E_j\psi_j.\\
 &\sum_j  |a_j|^2 \leq c<\infty.
 \end{align}
 Here stability of the bound cluster means that the $a_j$ are time independent.

 \begin{theorem}\label{AC1-proof}
 Let $H,V$ as above. Then for all $\psi \in L^2(\R^n),$  AC (\ref{AC-linear}) holds.
 \end{theorem}
 \noindent{\bf Proof}:  The solution can be split to two parts, by projecting on the pure point spectral part of $H$ and its orthogonal complement.
 By Weyl's theorem, since the potential is relatively compact w.r.t. $H$, the essential spectrum is $[0,\infty].$
 We therefore are left with showing that the orthogonal complement of the bound states is scattering to a free wave.
 To this end, we note that due to the linearity of the problem, it is sufficient to prove this result for a dense set, in the $L^2$ sense, of the orthogonal complement of the bound states.
 We choose the dense set to be the range of $F_0( \eta\leq H \leq c)\psi, \psi \in L^2, \forall c,\eta >0.$
 To prove that vectors on the range of these operators converge to a free wave, we will prove that the corresponding channel wave operators exist:
 \begin{align}\label{Wave-O-linear}
 &s-\lim_{t \to \infty} e^{-i\Delta t}F_1e^{-iHt}F_0(\eta)\psi\\
 &w-\lim_{t \to \infty} e^{-i\Delta t}(I-F_1)e^{-iHt}F_0(\eta)\psi=0.\\
 &F_1\equiv F_1(\frac{|x|}{t^\alpha}\geq 1), \alpha \in (1/3, 1).\\
 &F_0(\eta)\equiv F_0( \eta\leq H \leq c) .
 \end{align}
 The function $F_1$ is a smooth projection on the argument domain. That is, it is equal to $1$ on the region of space $\frac{|x|}{t^\alpha}\geq 1.$

 To prove that the limit exists on support of $F_1$, we use Cook's method.
  We need to show that the derivative w.r.t. time is integrable in $t.$
 The derivative is
 \begin{align}
&\frac{\partial}{\partial t}e^{-i\Delta t}F_1e^{-iHt}F_0(\eta)\psi\\
= &e^{-i\Delta t}\big [t^{-\alpha}F'_1[2\gamma-\alpha |x|/t]+F_1 V\big]e^{-iHt}F_0(\eta)\psi+\mathbb{O}(t^{-2\alpha} F''_1)
\end{align}
The vector-field defining $\gamma$ is taken to be smooth and is equal to $x/|x|, |x|\geq 2.$
The term $F_1 V$ decays like $t^{-\sigma\alpha}.$ To cover the most general short range decay, we need to take $\alpha$ close to 1, which we can.
However, we take $\alpha > 1/3,$ since it is the minimum allowed for the argument to go through, in any dimension.

 Assume that the $F_1 V$ decays faster than $1/t$ and is therefore integrable.

Next, to control the $F'_1$ term, we need to prove a propagation estimate that implies the required integrability.
This propagation estimates is done with the following choice of propagation observable (PROB):
\begin{align}\label{PROB-1}
&\frac{\partial}{\partial t} \langle (1/2)[F_1\gamma+\gamma F_1]\rangle\\
=&\langle t^{-\alpha}[F'_1(2\gamma^2-(\alpha |x|/t)\gamma] + h.c.\rangle+\langle F_1L^2/r^3\rangle +  \mathbb{O}(t^{-1-\epsilon}).\label{dfgamma}
\end{align}
Here $L^2= |x\times p|^2$ represents the  angular momentum, which comes from the non-radial part of the solution.

To get the propagation estimate, we need to show that the RHS is a sum of a positive quantity and an integrable quantity(in time).
For this we use that:

\begin{proposition}\label{localization}
\begin{align}
\gamma^2F'_1 +F'_1\gamma^2= & 2\sqrt{F'_1}\gamma^2\sqrt{F'_1}+\mathbb{O}(t^{-2\alpha})=2\sqrt{F'_1}(-\Delta_r)\sqrt{F'_1}+\mathbb{O}(t^{-2\alpha}) \label{r2F}\\
(F'_1\alpha |x|/t)t^{-\alpha}\thickapprox & \alpha F'_1/t.\\
 \langle F'_1/t \gamma\rangle \lesssim & a\langle\sqrt {F'_1}/t^{\alpha} \gamma^2\sqrt {F'_1}\rangle+\langle F'_1/t^{1+\epsilon} \rangle\quad  a\ll 1.\label{Ft}\\
 \quad \quad F_1\tilde F_0(H\geq \eta)=&  F_1\tilde F_0(-\Delta\geq \eta)+\mathbb{O}(\eta^{-m-1}t^{-m\alpha}D^m),\quad D=|\partial|\label{H-Delta}\\
F_0(-\Delta\geq \eta)&[ \sqrt{F'_1}(-\Delta_r)\sqrt{F'_1}+F_1L^2/r^3]F_0(-\Delta\geq \eta)\geq \frac12 F_0 \tilde F_1 \eta F_0. \label{Positive}\\
\tilde {F_1}\thickapprox& F_1F'_1.
\end{align}

\end{proposition}

\begin{proof}

The first statement follows from symmetrization, by commuting the square root of $F'_1.$
Since
$$
F^2G+GF^2=2FGF+[F,[F,G]],
$$
and each commutator with $F$ gives a factor of $t^{-\alpha}, \alpha>1/3,$ the result follows.
This also explains why $\alpha >1/3$ is the border-line.
The third inequality is a consequence of Cauchy-Schwarz inequality.
To prove (\ref{H-Delta}), we use the following representation:
\begin{align}\label{H-H_0}
F_1\tilde F_0(H\geq \eta)= & F_1\tilde F_0(H\geq \eta)- F_1\tilde F_0(-\Delta\geq \eta)+F_1\tilde F_0(-\Delta\geq \eta)\\
=&
c\int_{-\infty}^{\infty} \mathfrak{F}{\tilde F_0}(u) F_1[e^{iHu}-e^{-i\Delta u}]du+F_1\tilde F_0(-\Delta\geq \eta)\\
F_1[e^{iHu}-e^{-i\Delta u}]= & F_1e^{-i\Delta u}\int_{0}^{u}e^{i\Delta s}iVe^{iHs}ds\\
= &F_1\jx^{-m} \mathbb{O}(D^m)e^{-i\Delta u}\int_{0}^{u}s^m e^{i\Delta s}i\jx^m Ve^{iHs}ds\\ =& F_1\mathbb{O}(t^{-m\alpha}D^m)u^{m+1} \|\jx ^mV(x)\|_{\infty}.
\end{align}
The result now follows by using that $\|u^{m+1} \mathfrak{F}{\tilde F_0}(u)\|_{L^1}\lesssim 1/\eta^{m+1}.$
\end{proof}
From (\ref{dfgamma})(\ref{r2F})(\ref{Ft}), we get that
\begin{equation}\label{PRES1}
\int_{1}^{T}\frac{1}{t^{\alpha}}\|\sqrt{-\Delta_r} \tilde{F_1}\psi(t)\|^2 dt \leq C_{\eta}\|\psi(0)\|^2_{H^{1/2}}.
\end{equation}
From Propposition~\ref{localization} and (\ref{H-H_0}),(\ref{PRES1}) we get
\begin{equation}\label{PRES2}
\int_{1}^{T}\frac{1}{t^\alpha}\|\eta \tilde{F_1}\psi(t)\|^2 dt \leq \sup_t\|\psi(t)\|^2_{H^{1/2}}+\sup_t \jap{D^m}\eta^{-m+1}
\end{equation}
We now take $\psi\in F_0(H\ge \eta)\bar{F_0}(H\leq c)L^2,$ and then $\jap{D^m}<\infty,$
$c$ arbitrary. Since
$$
D^2\approx -\Delta=H-V, \quad (H-V)\bar{F_0}(H\leq c)\leq (c+\mathbb{O}(1))\bar{F_0}(H\leq c).
$$
We also used that $V$ is $H$ bounded. 

(\ref{PRES2}) implies that the channel wave operator exists:
\begin{equation}\label{channel 1}
s-\lim e^{-i\Delta t}F_1(\frac{|x|}{t^{\alpha}}\geq1)e^{-iHt}g(H\geq \eta)\equiv \Omega ^*
\quad \quad \forall \eta \text{ positive}.
\end{equation}

$\Omega\Omega^*$ projects on initial states that are asymptotically free waves.
Therefore, any $\psi(0)\in \mathbb{H}_{ess}(H)$ decomposes as:
\begin{align}
&\psi(0)=\psi_{wls}+\psi_{free}\equiv \psi_{wls}+\Omega\Omega^* \psi(0),\\
&e^{-iHt}\psi_{wls}\simeq \bar{F_1}(\frac{|x|}{t^{\alpha}}\leq 1)e^{-iHt}\psi_{wls}.
\end{align}
The second one means $\| \bar{F_1}(\frac{|x|}{t^{\alpha}}\geq 1)e^{-iHt}\psi_{wls}\|_{L^2}\rightarrow 0$ as $t\rightarrow +\infty$.
It is now left to show that any $\psi_{wls}=0.$

 Formally, this follows from the Dilation Identity:
 \begin{align}\label{Dilation}
 \partial^2_t \langle x^2\rangle = & \partial^2_t (e^{-iHt}\psi_{wls},x^2 e^{-iHt}\psi_{wls})\\
 =& (e^{-iHt}\psi_{wls},(8(-\Delta)-4x\cdot\nabla V) e^{-iHt}\psi_{wls})\\
\geq& (e^{-iHt}\psi_{wls}, e^{-iHt}\psi_{wls})\eta-(e^{-iHt}\psi_{wls},-8V-4x\cdot\nabla V e^{-iHt}\psi_{wls}).
 \end{align}
 On the continuous spectral part of $\psi_{wls}$, the time average of the last term goes to zero with time.
 Upon two integrations over time, we therefore conclude that the LHS, $\langle x^2\rangle$ grows like $t^2.$
 This means that part of the solution moves in $L^2$ sense to a distance $t$.
 This is not possible for the WLS solution.
 As can be seen from the computation, we need to know that the solution is in the domain of $x.$
 Therefore, we need to approximate such states with a localized state.
 One can check that any vector in the Hilbert space, in the range of the operator $g(H\sim \eta)$,
 can be approximated by $g(H\sim \eta)\jx^{-2}\psi_{\epsilon}.$
 For this, one only needs to check that the commutator $[x,g(H)]$ is bounded for smooth $g.$
 Then we use that:
\begin{equation}
 \|F_1(\frac{|x|}{t^{\alpha}})e^{-iHt}\psi_{\epsilon}\|\leq  \|F_1(\frac{|x|}{t^{\alpha}})e^{-iHt}[\psi_{\epsilon}-\psi_{wls}]\|+
 \|F_1  e^{-iHt}\psi_{wls}\|\leq 2\epsilon,
 \end{equation}
 for all $t>T(\epsilon).$
 It remains to check that if $\langle x^2\rangle\geq ct^2,$ then we get a contradiction.
 Let $\psi_w(t)$ be a weakly localized solution. Then we have that
 \begin{equation}
 \|F_1(|x|/t^{\alpha})\psi_w(t)\|_{L^2} \rightarrow 0,
 \end{equation}
 as $t$ goes to infinity.
 Therefore,
 \begin{align}
 &\exists t_n,\quad
 \|F_1(|x|/t_n^{\alpha})\psi_w(t_n)\|_{L^2}\leq 1/n^2.\\
 &\|F_1(|x|/t^{\alpha})\psi_w(t)\|_{L^2} \leq \| F(\frac{|x|}{t}\leq M) F_1(|x|/t^{\alpha})\psi_w(t)\|_{L^2} +c/t.\\
 & M\geqq 1.\quad \|x F F_1(|x|/t_n^{\alpha})\psi_w(t_n)\|_{L^2}\leq Mt_n \, \|F_1(|x|/t^{\alpha})\psi_w(t)\|_{L^2}\leq Mt_n/n^2.
 \end{align}
 \begin{equation} 
 \langle \psi_w(t_n), x^2 \psi_w(t_n)\rangle \leq Mt_n^2/n^4.
 \end{equation}
 We also used the maximal velocity estimates, to insert $F(\frac{|x|}{t}\leq M).$
 On the complement there is decay, if the energy of the state is small compared with $M.$
 This follows from the following PROB:
 \begin{align}
 &D_H (x^2-M^2t^2)F_1(\frac{|x|}{t}\geq 2M)=\\
 &(4A-2M^2t)F_1 +(x^2-M^2t^2)F'_1(t^{-1}\gamma-t^{-2}|x|)+\text{h.c.}+\text{higher order}\\
 &\lesssim (4|x|E-2M^2t)F_1+(4M^2t^2-M^2t^2)F't^{-1}(\gamma-M)+\text{h.c.}+\text{higher order}\leq 0 +\text{h.c.}.
 \end{align}
 The resulting propagation estimate is the desired bound used above.
 Here $E$ is the maximal energy of the initial stat: $\psi(0)=F(H\leq E)\psi(0).$
 
 We therefore conclude that the continuous spectral part of $H$ does not have WLS (weakly localized states),
 which implies that the channel wave operator has the full continuous spectrum in its range.
 This implies that the spectrum is absolutely continuous.

 It is still possible that there are positive eigenvalues, but by a known compactness argument, it is a discrete set.
 It is possible to use arguments as above to exclude them, but this is beyond the scope of this article.
 \section{Nonlinear and Time-Dependent Potentials}
 Following the ideas of the above example of two body Quantum Hamiltonians, we explain the method to treat NLS type equations and time dependent potentials.

 First, we need to have localized perturbation terms. This is achieved by assuming that the initial data is radial, in $H^1(\R^3)$
 or higher dimensions. Then we can use the Radial Sobolev embedding theorems ($n=3$):
 \begin{equation}
 r|u(r)|\lesssim \|u(\cdot)\|_{H^1},\quad \quad r^{1/2}|u(r)|\lesssim \|u(\cdot)\|_{\dot{H}^1} \text{ for }\quad r\geq 1.
 \end{equation}
 We then choose the nonlinearities/potentials such that the $H^1$ norm is uniformly bounded. Furthermore, we require the nonlinearity
 to vanish to a sufficient order in $u$ for $u$ near zero. Consequently, recall $\mathcal{N}(\psi)=V(|x|,t, |\psi|)$
\begin{equation}
  F_1V=\mathbb{O}(t^{-1-\epsilon})\|\psi\|_{H^1}^m.
\end{equation}

 Next, we need to deal with the fundamental difference from the time independent case, which is the inability to localize the energy away from zero, as we did with functions of $H.$

 Indeed, the dynamics can squeeze the solution to zero frequency as time goes on, producing a new channel of asymptotic states.
 We refer to these kind of states, if they exist, as \emph{weakly localized states}. As an example one should consider self similar solutions.
 The way to go around this problem, is to localize $\gamma$ on the support of $F'_1$ so as to insure getting positive commutators.
 That means we will use propagation estimates derived from PROB of the form:
 \begin{align}\label{F1F2}
& B_1 \equiv F_1(|x|/t^{\alpha})F_2(\gamma)F_1(|x|/t^{\alpha})\\
&F_2=F_2(y\geq \epsilon),F_2(y\leq -\epsilon)\\
&B_2 \equiv (|x|/t^{\alpha})F_1(|x|/t^{\alpha})F_2(\gamma)+h.c.
\end{align}
\Red{}
The reason we get good estimates from $B_1,B_2$ is the following:
\begin{align}
D_H (F_1F_2F_1)= & t^{-\alpha}\big[F'_1[2\gamma-\alpha|x|/t]F_2F_1+F_1[2\gamma-\alpha|x|/t]F_2F'_1\big]+\mathbb{O}(t^{-1-\epsilon})\\
\geq& t^{-\alpha}\sqrt{F'_1}\epsilon F_2\sqrt{F'_1}+\mathbb{O}(L^1(dt))\\
[F_1,\partial_x]= & \mathbb{O}(t^{-\alpha}),
\end{align}
Here we used that
$ D_H F_2=0$ in the radial case.
The resulting PRES imply the existence of channel wave operators, as before, and in particular imply that the states propagating on the range of $B_1, B_2$ become free waves ( and zero when $F_2=F_2(y\leq -\epsilon)$ ).

In the TD/NL (Time-Dependent, Nonlinear) potential case, it is not possible to exclude arbitrary small frequency from scattering.
The interaction terms can create another channel of asymptotic behavior, by focusing at frequency zero.
To this end, we need a geometric decomposition to separate the free waves from the rest, which are only \emph{weakly localized states} (WLS) in general.
For this we use the following :
\begin{theorem}
The limit
$$\lim_{t \to \infty} \langle F_1(|x|/t^{\alpha})\gamma F_1(|x|/t^{\alpha})\rangle= \Gamma$$
exists for all $\alpha\in(1/3,1)$, independently of $\alpha.$ If  $\Gamma>0$ the asymptotic state contains a free wave with non-zero $L^2$ norm.
 If $\Gamma<0$ the solution is zero; the exception is when the solution blows up in a finite time.
 If $\Gamma =0$ the solution is a WLS.

 Moreover, any WLS satisfies the following bound:
 \begin{equation}\label{WLS}
 \langle |x|\rangle \leq ct^{1/2}, t\geqq 1
 \end{equation}
\end{theorem}

\begin{remark}
We assume the solutions are uniformly bounded for all times, so the case
$\Gamma<0$ is excluded.
\end{remark}
\begin{remark}
The surprising fact that for general nonlinearities and potentials the bound \ref{WLS} holds,
can be understood by invoking the notion of propagation set~\cite{SSAnnals, SSInvention}. 
If the solution propagates faster than $t^{1/2+0}$, then the propagation set is where $\gamma \geq t^{-1/2+0}.$
 The Heisenberg derivative $D_H F_1$ is positive there for $1>\alpha >1/2$ since  $2\gamma-\frac{|x|}{t}\sim t^{-\beta} -\alpha t^{\alpha-1},$
 and the higher order terms coming from symmetrization are of order $t^{-3\alpha+\beta}, \beta=1-\alpha.$
The error term is integrable if $\alpha >1/2.$
One then get a propagation estimate which implies that this part of the solution becomes a free wave.
\end{remark}

This analysis and the problems described above necessitates microlocalizing $\gamma$ in sets of shrinking size with time.
We are then led to consider the second microlocalized version of the above $B_1.$

We then introduce PROB of the form
$$
F_1(|x|/t^{\alpha})F_2(t^{\beta}\gamma\lessgtr 1)F_1(|x|/t^{\alpha}).
$$
In deriving the corresponding PRES from the above operators we see that each commutator with $F_1$ we gain a factor of $t^{-\alpha},$
but each commutator (with $x$)of $F_2$ we lose a factor of $t^{\beta}.$
This limits the choices of such PROB.
Another issue that comes up is the regime of phase-space which is second microlocalized, \emph{on the propagation set}.
This happens when $x\thicksim 2\gamma t.$ In this region is it difficult to fix the sign of the Heisenberg derivative of $F_1.$
The above analysis implies the first key part of AC: that all solutions in the radial case, converge in $L^2$ to a linear combination of a free wave and a WLS.

\subsubsection{Weakly Localized states   }

In the nonlinear case the understanding of the weakly bound states is very complicated.
There are many possibilities for the solution to be weakly localized. It can be fully localized, like a soliton or a breather, or it can spread as a self similar solution, or asymptotically self similar.
Even if one can prove that it is localized in the sense that
$$
|\jx^a\psi(t)| \lesssim 1, \quad \forall t,
$$
which excludes self similar or any other spreading solutions, it is not clear what is the time dependence.
At best, we have the following conjecture, with some partial results:

\begin{definition}  Following the definition in~\cite{Soffer}, we  call a solution $u(t,x)$ to equation  (\ref{Main-eq}) a \textbf{bound state} if it satisfies
\begin{equation} \lim_{R\rightarrow \infty} \sup_{t\in[1,\infty)}\|\chi({|x|>R})u(t,x)\|_{H^1(\R^3)}=0,\end{equation}
where $\chi(x)$ is a characteristic function.
\end{definition}

\noindent {\bf  Petite Conjecture: }
Bound states are
 (asymptotically) almost periodic functions of time.
 \smallskip
 \begin{remark}
 Our analysis shows that time periodic, and by extension almost periodic solutions (in time) which are WLS, are in fact localized.
 So, the converse of the Petite Conjecture follows.
 Since we can also show that the WLS are smooth, most likely it reduces the proof of the Petite Conjecture to excluding arbitrarily slow oscillations in time.
 \end{remark}

 Our analysis of the WLS is based on proving propagation estimates, from which we derive properties of such solutions.
 Some notable conclusions are:
  WLS are smooth, and moreover, the derivatives $\partial_x^m\psi(t),m\geq 1 $ are localized in space, uniformly in time.

  First we have the regularity result of WLS. We sketch the proof below.

  Assume first that we proved already      that the derivative of the solution is localized.
  In fact we only prove that the derivative is in the domain of $x$, but a decomposition to compact annular domains reduces the problem to compactly supported cases. Then, we have:

  \begin{proposition}[Improvement of regularity]
Let $\psi(t)$ be a global solution to (\ref{Main-eq}) satisfying the global energy bound (\ref{global-bound}),  and it is strictly localized, i.e.
 $supp \psi(t) \in B_K(0)$ for all $t\geq 0$.
 Then $\psi_0\in C^\infty$.
\end{proposition}
We sketch some key steps:
  \begin{proof}
  Consider denote $\nabla \mathcal{N}(\psi) = \tilde{\mathcal{N}}(\psi)\nabla\psi$, we differentiate the equation (\ref{Main-eq}), and get the Duhamel formula for $\nabla \psi$
  \begin{equation}
  \nabla \psi(t) = e^{it\Delta} \nabla \psi(0) -\int_0^te^{i(t-s)\Delta} \tilde{\mathcal{N}}(\psi)\nabla\psi(s)ds
  \end{equation}
  Let $M\gg K>1$, since  $supp \psi(t) \in B_K(0)$ for $|t|\leq 1$. Denote $P_M$ the Littlewood-Paley operator that projects onto frequency $M$. Then by minimal velocity bound, at time $t=M^{-\frac12}$, $P_M  e^{it\Delta} \nabla \psi(0) $ is essentially supported outside the ball of $B_{K+M^\frac12}(0)$. Since $\psi(t)$ vanishes on $B_{K+M^\frac12}(0)^c$,
  take $\chi(x)$ to be the characteristic function of $B_{K+M^\frac12}(0)^c$,  we get
  \begin{align}
  \|\chi(x)P_M  e^{iM^{-\frac12}\Delta} \nabla \psi(0) \|_{L^2} =& \|\chi(x)\int_0^{\frac{1}{\sqrt{M}}}e^{i(t-s)\Delta} \tilde{\mathcal{N}}(\psi)\nabla\psi(s)ds\|_{L^2}
  \lesssim \frac{1}{\sqrt{M}}
  \end{align}
 \begin{align}
\|P_M   \nabla \psi(0) \|_{L^2}  = \|\chi(x)P_M  e^{iM^{-\frac12}\Delta} \nabla \psi(0) \|_{L^2}  + \|(1-\chi(x))P_M  e^{iM^{-\frac12}\Delta} \nabla \psi(0) \|_{L^2}
 \end{align}
  We only need the minimal velocity bound
  \begin{equation} \|(1-\chi(x))P_M  e^{iM^{-\frac12}\Delta} \nabla \psi(0) \|_{L^2}  \lesssim \frac{1}{\sqrt{M}}\end{equation}
  Hence by taking dyadic $M_k=2^k M_0,  M_0\gg K^2>1$, we have  for $\alpha\in(0,\frac12)$
  \begin{align}
\||D|^{1+\alpha}\psi_0\|^2_{L^2} =\sum_{k=0}^\infty  \|M_k^\alpha P_{M_k}\nabla \psi(0)\|^2_{L^2}\lesssim \sum_{k=0}^\infty M_k^{2\alpha-1}\lesssim 1.
\end{align}
Therefore, if $\psi, \nabla \psi$ are supported in $B_K(0)$ and $\psi\in H^1$ uniformly for $|t|\leq 1$, (so this is in fact weaker assumption then currently stated in the prop)   and $sup_{t\leq 1}\| \tilde{\mathcal{N}}(\psi)\|_{L^\infty}\lesssim 1$. This assumption holds for example
$\mathcal{N}=\frac{|\psi|^m}{1+\delta |\psi|^m}V(x,t), \delta>0, m>m_0\geq 3$.  $V, \nabla V \in L^\infty_{t,x}$, $\psi$ radial.  And assume $V(x,t)\psi, \nabla \psi$ is localized in $K$ for $t\leq 1.$

Now we get  $\alpha$-improvement of regularity. Then we can iterate the argument by taking higher order derivatives on the equation.
\end{proof}

  \subsection{Exterior Morawetz Estimate- Localization    }

\begin{proposition}For weakly localized solution $\psi(t)$, there exists a sequence of time $t_n\rightarrow+\infty$, such that
\[\|A\psi(t_n)\|_{L^2}\lesssim1.\]
\end{proposition}
\begin{proof}
Consider the propagation observable $F_1(\frac{|x|}{M}\geq 1)\gamma F_1(\frac{|x|}{M}\geq 1)$ for $M\geq 1.$ The commutator with the nonlinearity is assumed to be of order $M^{-k}$, $k$ large.  The leading term of the Heisenberg derivative comes from the Laplacian, 
hence  we get (notice the $[-i\Delta ,\gamma]$ vanishes on the support of $F_1$)
\begin{align}
D_H  F_1(\frac{|x|}{M}\geq 1)\gamma F_1(\frac{|x|}{M}\geq 1) =  \frac{1}{M}    \sqrt{F_1F_1'}  \gamma^2  \sqrt{F_1F_1'}   + F_1[-i\Delta ,\gamma] F_1+ \frac{1}{M^3}\tilde{F}_1(\frac{|x|}{M}\sim 1).
 \end{align}
 Hence we obtain the estimate
 \begin{align}
 \la F_1 \gamma F_1 \rangle_{t} -  \la F_1 \gamma F_1 \rangle_{1} =\int_1^t \frac{1}{M }  \la   \sqrt{F_1F_1'}  \gamma^2  \sqrt{F_1F_1'} \rangle  +\frac{1}{M^3} \la \tilde{F}_1(\frac{|x|}{M}\sim 1)\rangle +O(M^{-k}) ds
 \end{align}
 For weakly localized state, we have
 \begin{equation} \int_1^T \la \psi, \gamma \psi\rangle ds \lesssim \sqrt{T}\end{equation}
 and
 \begin{align}
 \partial_t \la  F_1|x|F_1\ra =\la  F_1 \gamma F_1\ra + \frac{1}{M}\la  \tilde{F}_1\gamma \tilde{F}_1\ra  + \frac{1}{M^3}\la \tilde{F}_1\ra
\end{align}
Here $\tilde{F}_1=\tilde{F}_1(\frac{|x|}{M}\sim 1)$, hence we have
\begin{equation} \int_0^T \la  F_1 \gamma F_1\ra \leq C\sqrt{T}+\frac{T}{M^3}.\end{equation}
Therefore we get the estimate
\begin{equation} \frac{1}{MT}\int_1^T  \la  \sqrt{F_1F_1'}  \gamma^2  \sqrt{F_1F_1'} \ra dt \lesssim  \frac{1}{T^{\frac32}} +\frac{1}{TM^3}\int_1^T \la \tilde{F}_1\ra dt .  \end{equation}
Therefore, on average and as $T\rightarrow \infty$, for typical $t_n$, we have
\begin{equation}\|\gamma  \sqrt{F_1F_1'}  \psi(t_n)\|_{L^2}^2\lesssim \frac{M}{T^{\frac32}} +\frac{1}{M^2} \|\tilde{F}(\frac{|x|}{M}\sim 1)\psi(t_n)\|_{L^2}^2.\end{equation}
Multiplying by $M^2$ on both sides, and divide by $(ln M)^{1+\epsilon}$, and let $M=M_02^k, k\in \mathbb{N}$, then take a sum up to $M_02^k=\sqrt{T}$. ($T\sim t_n$),
we see that $\|\bar{F} x\cdot \nabla \psi(t)\|_{L^2}\lesssim 1$, here $\bar{F}=F(\frac{|x|}{\sqrt{t}}\leq 1)\frac{1}{\ln\jx}$.  For $M\geq  t^\alpha, \alpha>\frac13,$ we have that
\[\la F_1 \gamma F_1\ra \rightarrow 0, \hspace{1cm} t\rightarrow +\infty.\]
Hence if $\la F_1 \gamma F_1\ra_1>0$, we have $\la F_1 \gamma F_1\ra_t - \la F_1 \gamma F_1\ra_1\leq 0$ for $t >t_1$.

So for $M\geq t^\alpha, \alpha>\frac13, $ we have  for $T$ large, (Here the notation $G$ is in fact our $\sqrt{F_1F_1'}$ in previous estimate)
\begin{equation}\int_1^T\|G(\frac{|x|}{M}\sim 1)\nabla \psi\|^2 dt \leq \frac{1}{M^2} \int_1^T \|G\psi\|^2\end{equation}
So for all $M$, we have the bound
\begin{equation}\|x\cdot \nabla \psi(t_n)\|_{L^2}\lesssim 1\end{equation}
 \end{proof}

  To motivate the construction of other propagation observables, we first show how one can estimate the contribution of the nonlinearity in various parts of the phase-space.

  We saw before that on the support of $F_1(|x|/t^{\alpha})$ one uses the radial symmetry, $H^1$ bound and Radial Sobolev Embedding, to get decay of the nonlinearity for large $r$ and therefore decay on the support of $F_1(|x|/t^{\alpha}).$

  This argument also applies when $|x|/t^{\alpha}$ is replaced by $|x|/M, M\geqq 1.$ In this case the higher order terms are not integrable in time, but higher order in $M$. Estimates follow from summing over $M\equiv 2^nM_0.$

  Another way to get decay of the nonlinear term is based on projecting on the region of phase-space where $A$, the Dilation operator is large.
  In this case we use that
\begin{equation}
  F(A\geq M)\jap A^{-2m}(A^2+1)^m\jx^{-2m}\jap D^{-2m}=\mathbb{O}(M^{-2m}).\label{AgM}
\end{equation}
  The decay in $x$ then comes from powers of $\psi$ and/or decay of the potential.
  The control of the derivatives comes from the regularity of the potential and that of $\psi.$
  In particular, when $\psi$ is WLS, we have  more regularity, and localization.

  Still another way, is to localize in the regime of small frequency: $F_p(t^{\beta}|p|\leq 1), \beta>1/3.$
  In this case, one can use the localization properties of the nonlinearity, together with Hardy-Littlewood-Sobolev estimates
  to prove that the nonlinear term is decaying integrably in time (when the dimension is 3 or higher, and the nonlinearity vanishes to sufficient order in $\psi$).
  This implies that the part of the solution that focus into the support of $F_p$ is an independent asymptotic channel of propagation, if it is not empty. Since such a state has zero energy at infinite time, any initial condition that propagates into this channel, has initial energy $E=0$ due to energy conservation.

  \subsection{New propagation estimate- nonlinear case}
First recall the result from~\cite{Soffer-monotonic}, If $A$ is the dilation on $L^2(\R^n)$, and $\tanh (A/R)$ is defined for $R>\frac{2}{\pi}$, then
\begin{equation} [-i\Delta, \tanh(\frac{A}{R})]=p\frac{1}{R ch^2(\frac{A}{R})}p\end{equation}
Moreover, we have for an analytic function $F$ in a sufficiently wide strip around the real axis, that
\begin{align}
i[p, F(A) ]=&  i F(A-i)p-F(A)p = i p[F^*(A+i)-F^*(A)]\\
i[x, F(A) ]=&  i F(A+i)x-F(A)x = i x[F^*(A-i)-F^*(A)]
 \end{align}

 Therefore,  for $P_{M,R}^{+}(A) =\frac12 (1+\tanh(\frac{A - M}{R}))$ (similarly the incoming projection  $P_{M,R}^{-}(A) =\frac12 (1-\tanh(\frac{A + M}{R}))$)
\begin{equation} [-i\Delta,P_{M,R}^{+}(A)]=P\frac{1}{R ch^2(\frac{A-M}{R})}P\end{equation}
This implies that for propagation observable $B=AP_{M,R}^{+}(A)$, we have
\begin{align}
\partial_t \la B\ra =\la p P^+p\ra +\la p\frac{A}{R ch^2(\frac{A-M}{R})}p \ra + \la [iV(x,t), AP_{M,R}^{+}(A) ]\ra
\end{align}
For general time dependent potential or nonlinearity $V$, the first two terms are positive modulo corrections that are exponentially small. If $M\gg 1, R=\sqrt{M}, $ then the second term is negative for $A\leq 0 $ or $A-M\leq -M$. But then
\begin{equation} \left|\frac{A}{R ch^2(\frac{A-M}{R})}\right|\lesssim |\frac{A - M}{R}|e^{-2(\frac{A-M}{R})}\end{equation}
for $|A-M|\geq M$, so the largest value is when $|A_M|=M$, which implies that  when taking $R=\sqrt{M}$,
\begin{equation} \left|\frac{A}{R ch^2(\frac{A-M}{R})}\right|\lesssim |\frac{M}{R}|e^{-2(\frac{M}{R})}\lesssim \sqrt{M} e^{-2\sqrt{M}} \end{equation}
So we ignore for the moment the exponentially small corrections,

It remains to estimate $ \la [iV(x,t), AP_{M,R}^{+}(A) ]\ra$
\begin{align}
i[V, AP^+ ] = i[V, A]P^+ +iA [V, P^+]= -x\cdot \nabla VP^++iA [V, P^+]
\end{align}
 This is done by the argument explained above using (\ref{AgM}).

 Finally, we conclude that : if
$|\la \psi(t_n), AP^+ \psi(t_n) \ra|\leq C$, then
\begin{align}
\frac{1}{T}\int_0^T \la \psi(s), P_{M,R}^+(A)\psi(s)\ra ds\lesssim & O(M^{-2m}) +\frac{C}{T}\\
\frac{1}{R}\int_0^T \la P\psi(s), \frac{|A|}{ch^2(\frac{(A-M)}{R})P\psi}\ra ds \leq& O(e^{-\sqrt{M}} +O(M^{-2k}))T+const
\end{align}
for all T.
This is then used to show that WLS satisfy
 \begin{equation}
 \la |A|\ra \lesssim 1.
 \end{equation}

\end{document}